\documentclass[journal]{IEEEtran}  

\IEEEoverridecommandlockouts                              

\usepackage{graphics}
\usepackage{graphicx} 
\usepackage{rotating}
\usepackage{dcolumn}
\usepackage{longtable}
\usepackage{multirow}
\usepackage{amsmath}
\usepackage{bm}

\usepackage{amsthm}
\usepackage{amssymb}
\usepackage{xspace}
\usepackage{textcase}
\usepackage{float}
\usepackage{dsfont}
\usepackage[latin1]{inputenc}
\usepackage{tikz}

\usetikzlibrary{trees}
\usetikzlibrary{arrows}

\newtheorem{theorem}{Theorem}[section]
\newtheorem{corollary}{Corollary}[theorem]
\newtheorem{lemma}[theorem]{Lemma}

\newcommand{\mc}{\mathcal}
\newcommand{\C}{\mathbb{C}}
\renewcommand{\Re}{\mathbb{R}}

\newcommand{\norm}[1]{\left\lVert#1\right\rVert}
\newcommand{\relambda}{\text{Re}\left(\lambda\right)}

\title{\LARGE \bf
Learning Bounded Koopman Observables: Results on Stability, Continuity, and Controllability
}

\author{Craig Bakker, Thiagarajan Ramachandran, and W. Steven Rosenthal
\thanks{C. Bakker, T. Ramachandran, and W.S. Rosenthal are with the Pacific Northwest National Laboratory,
        Richland, Washington
        {\tt\small craig.bakker@pnnl.gov}}%
}

\begin{document}

\maketitle
\thispagestyle{empty}
\pagestyle{empty}

\begin{abstract} 
The Koopman operator is an useful analytical tool for studying dynamical systems -- both controlled and uncontrolled.  For example, Koopman eigenfunctions can provide non-local stability information about the underlying dynamical system.  Koopman representations of nonlinear systems are commonly calculated using machine learning methods, which seek to represent the Koopman eigenfunctions as a linear combinations of nonlinear state measurements.  As such, it is important to understand whether, in principle, these eigenfunctions can be successfully obtained using machine learning and what eigenfunctions calculated in this way can tell us about the underlying system.  To that end, this paper presents an analysis of continuity, stability and control limitations associated with Koopman eigenfunctions under minimal assumptions and provides a discussion that relates these properties to the ability to calculate Koopman representations with machine learning.
\end{abstract}

\section{INTRODUCTION}

\subsection{THE KOOPMAN OPERATOR}

The Koopman Operator (KO) provides a way to transform a (potentially) nonlinear finite-dimensional dynamical system into an infinite-dimensional linear system.  It does this by lifting the nonlinear state dynamics into a functional space of observables, where the dynamics are linear \cite{budisic12jsr}.  Analytical Koopman representations of nonlinear systems are rare, which has motivated the use of data-driven methods -- in particular, the use of time series data to calculate finite truncations of the Koopman operator and its associated observables.  A common set of approaches for doing this is based on Dynamic Mode Decomposition (DMD) \cite{tu14jsr}.  DMD works by defining Koopman eigenfunctions as linear combinations of state variable measurements.  Extended DMD (EDMD) operates in a similar fashion but uses nonlinear functions of the state measurements \cite{williams15jsr}.  

The ability to represent Koopman observables and eigenfunctions that are nonlinear functions of the state can result in greater accuracy, but it also requires choosing a good dictionary of functions from which to work \cite{williams15jsr,kutz16jsr}.  Common choices include sets of polynomials \cite{williams15jsr,proctor18jsr} and radial basis functions \cite{williams16cp,korda18jsr}.  Dictionary-based approaches also suffer from combinatorial explosion as the dimension of the state increases.  An alternative to this uses neural networks to learn both the observables and the Koopman operator simultaneously \cite{yeung17jsr,li17jsr}.

Consider the autonomous $d$-dimensional dynamical system 

\begin{gather} \label{sys}
\dot{x} = f\left(x\right), \ x \in X \subseteq \Re^d
\end{gather}

where $f:\Re^d \rightarrow \Re^d$ is sufficiently smooth to guarantee the existence and uniqueness of solutions. We denote the flow induced by the system by $F^t(x)$ (i.e., $x(t) = F^t(x_0)$ is the solution to (\ref{sys}) at time $t$ starting from the initial condition $x_0 \in \Re^d$ at time 0).  Let $X \subset \Re^d$ be a compact set which is forward invariant under $F^t(\cdot)$.  The Koopman operator $\mc{K}^t$ describes the evolution of observables $g: X \rightarrow \C$ along the trajectories of (\ref{sys}):

\begin{gather} \label{def:Koop}
\mc{K}^t g = g \circ F^t. 
\end{gather}

The system (\ref{sys}) may be nonlinear, but the Koopman operator (\ref{def:Koop}) is always linear, so it can be characterized by its eigenvalues and eigenfunctions. A function $\phi: X \rightarrow \C$ is said to be an eigenfunction of $\mc{K}^t$ with eigenvalue $\lambda \in \C$ if 

\begin{gather} \label{Koop:eig}
(\mc{K}^t \phi)(\cdot) = e^{\lambda t}\phi(\cdot).
\end{gather}

Furthermore, for the Koopman semigroup $\{\mc{K}^t\}_{t \geq 0}$, we define $D(\mc{L})$ to be the set of all $g(x)$ such that the limit

\begin{align}
\mc{L} g = \lim_{t \rightarrow 0} \frac{\mc{K}^tg(x) -g(x)}{t}
\end{align}

exists in the sense of strong convergence. That is,

\begin{align}
\lim_{t \rightarrow 0} \norm{\mc{L}g - \frac{\mc{K}^t g-g}{t}}_{L^\infty} = 0.
\end{align}

The operator $\mc{L}$ is called the infinitesimal generator of the Koopman semigroup. The infinitesimal generator satisfies the eigenvalue equation

\begin{align}
\mc{L}\phi = \lambda \phi.
\end{align}

Additionally, if $g$ is continuously differentiable, we obtain

\begin{align}\label{chain_rule}
\mc{L}g = f \cdot \nabla g,
\end{align}

\noindent where $\nabla$ denotes the gradient.

\subsection{KOOPMAN OPERATOR THEORETICAL RESULTS}

Several papers address the theoretical considerations underlying data-driven methods for calculating KO representations.  Tu et al. \cite{tu14jsr} introduce the concept of linear consistency and show how it related to DMD's ability to calculate KO eigenvalue/eigenfunction pairs.  Arbabi and Mezi{\'c} \cite{arbabi17jsr} prove that DMD will converge to KO eigenvalues and eigenfunctions, for ergodic systems, as the time window of observations becomes infinitely long.  Budi{\v{s}}i{\'c} et al. also provide a convergence proof for KO mode calculation with generalized Laplace analysis and discuss the nature of generalized eigenfunctions for repeated KO eigenvalues \cite{budisic12jsr}.

If it is possible to calculate accurate KO representations numerically, it is then interesting and useful to study the connections between properties of the underlying dynamical systems and the KO representation.  We can first consider how known properties of the underlying dynamical system imply certain characteristics of the Koopman system.  For example, using notions of conjugacy developed in \cite{budisic12jsr}, Mezi{\'c} \cite{mezic15cp} claims that, if a nonlinear dynamical system is globally conjugate to a linear system, the KO spectrum can be determined from the spectrum of the dynamical system's Jacobian at a critical point.  Mauroy et al. \cite{mauroy13jsr} and Mauroy and Mezi{\'c} \cite{mauroy16jsr} provide some additional details and proofs; a consistent requirement therein is that the eigenvalues be distinct.

Tu et al. \cite{tu14jsr} also claim, without proof, that in dynamical systems with multiple basins of attraction, KO eigenfunctions are typically only supported on one basin of attraction.  The authors refer to the Duffing oscillator results in Williams et al. \cite{williams15jsr}, who in turn refers to Mauroy et al. \cite{mauroy13jsr}, though it is not clear which part of that paper Williams et al. are referring to.

We can also consider what known properties of the KO representation imply about the underlying dynamical system.  A key result here relates to eigenfunctions that are constant along trajectories; these eigenfunctions have $\lambda=0$ in continuous-time and $\lambda=1$ in discrete-time systems.  The level sets of these functions partition the phase space into invariant sets \cite{mezic15cp,mezic05jsr}.  Known eigenfunctions can also be combined to form new ones.  If $\phi_1,\lambda_1$ and $\phi_2,\lambda_2$ are two pairs of eigenfunctions and eigenvalues, then in discrete-time systems, $\phi_1^r \phi_2^s, \lambda^r_1 \lambda^s_2$ is also an eigenpair \cite{budisic12jsr} and $\phi_1^r \phi_2^s, r \lambda_1 + s \lambda_2$ is an eigenpair in continuous-time systems \cite{mauroy16jsr}.  This can allow us to construct these partitioning eigenfunctions.

Mauroy et al. \cite{mauroy13jsr} also define the concept of isostables with reference to KO eigenfunctions.  These isostables constitute level sets of a special kind of Lyapunov function -- one which has a constant decay rate.  This shows a way in which the KO can be used to obtain non-local stability information.  Similarly, Mezi{\'c} \cite{mezic15cp} shows that if a nonlinear dynamical system is globally conjugate to a linear system, the KO eigenfunctions of that system can be used to identify stable, unstable, and center manifolds.

The most detailed theory paper of this kind is Mauroy and Mezi{\'c} \cite{mauroy16jsr}.  They focus on global stability for hyperbolic attractors (both fixed points and limit cycles), and a foundational assumption of the paper is that the observables (and therefore the eigenfunctions) under consideration are continuous.  The key results of the paper are that trajectories converge to sets where eigenfunctions with negative eigenvalues (or complex eigenvalues with negative real part) are 0 and that, under certain conditions, it is possible to prove global stability using the eigenfunctions corresponding to the eigenvalues of the original system's Jacobian at the fixed point.  There are analogous results for limit cycles.  The authors then compute KO eigenfunctions with polynomial bases and use these to estimate basins of attraction.

These results are important, but they come with some potentially restrictive assumptions.  In particular, the assumption of eigenfunction continuity may not hold for multi-modal nonlinear systems. The main contribution of this paper is to provide a collection of results that establish continuity, stability and control limits of Koopman representations under minimal assumptions (boundedness, primarily, as opposed to continuity).  Section \ref{sec:proofs} is a collection of theorems pertaining to stability, continuity and control. Section \ref{sec:discussion} then provides a brief discussion that connects the continuity and controllability results provided in Section \ref{sec:proofs} to the ability of machine learning methods to produce Koopman representations.

\section{ANALYTICAL RESULTS}
\label{sec:proofs}

\subsection{STABILITY AND CRITICAL POINTS}
\label{sec:stability}

\begin{lemma}
\label{fp zero lemma}
If a Koopman system $\dot{\psi} = \mathcal{L} \psi\left(x\right)$ for a dynamical system $\dot{x} = f\left(x\right)$ has an eigenvalue $\lambda \neq 0$ with associated eigenfunction $\phi\left(x\right)$, then

\begin{gather}
\left\{ x: f\left(x\right) = 0\right\} \subseteq \left\{ x: \phi\left(x\right) = 0\right\}
\end{gather}
\end{lemma}

\begin{proof}

For any $x$, 

\begin{gather}
\phi\left(F^t\left(x\right)\right) = e^{\lambda t} \phi \left(x\right)
\end{gather}

If $x^{\star}$ is a stationary point, then

\begin{gather}
F^t\left(x^{\star}\right) = x^{\star} \ \forall \ t
\end{gather}

\noindent where $F^t\left(x\right)$ is the flow mapping under $\dot{x} = f\left(x\right)$.  For this to hold, we must have

\begin{gather}
\phi\left(F^t\left(x^{\star}\right)\right) = e^{\lambda t} \phi\left(x^{\star}\right) \\
 \phi\left(x^{\star}\right) = e^{\lambda t} \phi\left(x^{\star}\right)
\end{gather}

Since $\lambda \neq 0$, $\phi\left(x^{\star}\right) = 0$.

\end{proof}

Note that this includes purely imaginary $\lambda$ and $\lambda$ such that $\relambda>0$, not just $\relambda < 0$.

\begin{corollary}
\label{cor:cor11}
If 

\begin{gather}
\lim_{x \rightarrow \infty} \phi \left(x\right) = 0
\end{gather}

\noindent then $\relambda < 0$ does not necessarily imply stability

\end{corollary}

The negative eigenvalue means that $\phi\left(x\left(t\right)\right) \rightarrow 0$ as $t \rightarrow \infty$, but that does not necessarily mean that $x\left(t\right)$ converges to a critical point.  In other words, a negative eigenvalue need not imply stability.

\begin{theorem}
If $\relambda > 0$ and $x^{\star}$ is a stable fixed point on some region $M$ such that

\begin{gather}
\lim_{t \rightarrow \infty} F^t\left(x_0\right) = x^{\star} \ \forall \ x_0 \in M
\end{gather}

\noindent and $\phi\left(x\right) > 0$ for $x \in M$, then either

\begin{gather}
\lim_{x \rightarrow x^{\star}} \left|\phi\left(x\right)\right| = \infty
\end{gather}

\noindent or $F^t\left(x\right)$ is only defined for $t \leq T$.

\end{theorem}

\begin{proof}

Since $\phi\left(x\left(t\right)\right) = e^{\lambda t} \phi\left(x_0\right)$, if $x\left(t\right)$ defined as $t \rightarrow \infty$, then

\begin{gather}
\lim_{t \rightarrow \infty} \left|\phi\left(F^t\left(x\right)\right) \right|
\end{gather}

\noindent is defined.  Therefore

\begin{gather}
\lim_{t \rightarrow \infty} \left|\phi\left(F^t\left(x\right)\right)\right| = \lim_{t\rightarrow \infty} \left|e^{\lambda t} \phi\left(x\right)\right| = \infty
\end{gather}

In order for $\left| \phi\left(F^t\left(x\right)\right)\right|$ to remain bounded, for each $x$, $t$ must not be allowed to go to $\infty$, and thus there exists some $T$ for which $F^t\left(x\right)$, $t > T$, is not defined.

\end{proof}

Just as a negative eigenvalue does not imply stability, a positive eigenvalue does not imply instability if the eigenfunction is unbounded or if the trajectory reaches the critical point in finite time.

\begin{corollary}
If $F^T\left(x\right) = x^{\star}$, then

\begin{gather}
\phi\left(x\right) = e^{-\lambda T} \phi\left(x^{\star}\right)
\end{gather}

\end{corollary}

This implies that if we start at a point to which trajectories converge in finite time, we can work backwards to calculate the value of the eigenfunctions on the fixed point's basin of attraction.

\begin{theorem}
\label{escaping thm}
If $\relambda > 0$ and $\phi \left(x\right)$ is non-zero and bounded on a closed region $M$, then there exists $T$ such that $F^t\left(x\right) \notin M$ for $t > T$, $x\in M$.
\end{theorem}

\begin{proof}

If $\left|\phi\left(x\right)\right| > 0$ and bounded, then $\exists \ \epsilon,C$ such that

\begin{align}
0 < \epsilon &\leq \left|\phi\left(x\right) \right| \leq C \ \forall \ x \in M \\
\phi\left(F^t\left(x\right)\right) &= e^{\lambda t} \phi\left(x\right) \\
T &\equiv \frac{1}{\lambda} \ln \left(\frac{C}{\epsilon}\right) \\
\left|\phi\left(F^T\left(x\right) \right)\right| &= e^{\lambda T} \left|\phi\left(x\right)\right| \nonumber \\
&= \frac{C}{\epsilon} \left|\phi\left(x\right)\right| \nonumber \\
&\geq \frac{C}{\epsilon} \epsilon = C
\end{align}

Therefore, $F^T\left(x\right) \notin M$.  Furthermore, if $t>T$, then 

\begin{align}
\left|\phi\left(F^t\left(x\right)\right)\right| &= \left|e^{\lambda t} \phi\left(x\right)\right| \nonumber \\
&> \left| \phi\left(F^T\left(x\right)\right) \right| = \left|e^{\lambda T} \phi\left(x\right)\right|
\end{align}

\noindent so $F^t\left(x\right) \notin M$.

\end{proof}

\begin{corollary}
\label{zero phi cor}
If $\phi\left(x\right)$ is bounded, $\relambda > 0$, and $F^t\left(x\right)$ is defined as $t \rightarrow \infty$, then $\phi\left(x\right) = 0$.  Conversely, if $\relambda < 0$ and 

\begin{gather}
\lim_{t \rightarrow -\infty} F^t\left(x\right) = x^{\star} \ \forall \ x \in M
\end{gather}

\noindent then $\phi\left(x\right) = 0 \ \forall \ x \ \in M$.
\end{corollary}

The implication of this Corollary is that if $\relambda > 0$, then $\phi\left(x\right) = 0$ on a critical point's stable manifold, and if $\relambda < 0$, then $\phi\left(x\right) = 0$ on a critical point's unstable manifold as long as $\lim_{t \rightarrow \pm \infty} F^t\left(x\right)$ defined and $\phi\left(x\right)$ is bounded.

\begin{corollary}
If $M$ contains a fixed point $x^{\star}$, and if $\left|\phi\left(x\right)\right| > 0$ on $M / x^{\star}$ and bounded on $M$, then $\relambda > 0$ implies that $x^{\star}$ is unstable on $M$.
\end{corollary}

Theorem \ref{escaping thm} and its associated corollaries essentially provide the conditions under which $\relambda > 0$ can be used to imply instability without relying on continuity assumptions.

\begin{theorem}
\label{lyapunov thm}
If the set $M\left(c\right) = \left\{ x: \left| \phi\left(x\right)\right|\leq c\right\}$ is closed and $\relambda \leq 0$, then any trajectory that enters $M\left(c\right)$ remains in $M\left(c\right)$.
\end{theorem}

\begin{proof}
Assume that a trajectory that starts in $M\left(c\right)$ leaves $S\left(c\right)$.  This would require there to exist $x\left(t_1\right) \in M\left(c\right)$ such that $\left|\phi\left(x\left(t_1\right)\right)\right| \leq c$ and $x\left(t_2\right) \notin M\left(c\right)$, where $t_2 > t_1$ such that $\left|\phi \left(x \left(t_2\right)\right)\right| > c$.  However, $\left|\phi\left(F^t\left(x\right)\right)\right|$ is always non-increasing.  Therefore

\begin{gather}
\phi\left(x\left(t_1\right)\right) \geq \phi\left(x\left(t_2\right)\right), \ t_1 \leq t_2
\end{gather}

\noindent which is a contradiction.


\end{proof}

Theorem \ref{lyapunov thm} then provides conditions under which $\relambda < 0$ implies stability, again, without assuming eigenfunction continuity.







\begin{theorem}
If $0 < \epsilon \leq \phi\left(x\right) \leq C$ for $x \in M$, $\relambda \neq 0$, and $\lim \limits_{t\rightarrow \infty} F^t\left(x\right)$ is defined, then any trajectory that enters $M$ will exit $M$.
\end{theorem}

\begin{proof}
If $\relambda > 0$, then

\begin{gather}
\lim_{t\rightarrow \infty} \phi\left(F^t \left(x\right)\right) = \infty \notin M
\end{gather}

Since $\phi\left(x\right)$ is bounded away from $0$ for $x \in M$, this holds.  Similarly if $\relambda < 0$, 

\begin{gather}
\lim_{t\rightarrow \infty} \phi\left(F^t \left(x\right)\right) = 0 \notin M
\end{gather}

\end{proof}

Most machine learning based approaches learn Koopman representations by training on time-series data and tend to learn the exact discretization. It can easily be shown that for exact discretization

\begin{gather}
x_{t+1} = x_t + \int\limits_0^{\Delta t} f\left(F^{\tau}\left(x_t\right)\right) d \tau = h \left(x_t\right)
\end{gather}

\noindent the negative eigenvalues get mapped to the interior of the unit circle, positive eigenvalues get mapped to the exterior, and zero eigenvalues get mapped to the boundary of the unit circle. This preserves all of the stability properties.  As such, the stability results proven for continuous time systems in this section carry over to their discrete time equivalents learned from the time-series data. 

\subsection{CONTINUITY}

The stability proofs did not assume continuity, but it is worth looking into continuity in more detail to see when it holds and when it does not.

\begin{lemma}
\label{basin invariance}
If $\lambda = 0$ and

\begin{gather}
\lim_{t \rightarrow \infty} F^t\left(x\right) = x^{\star} \ \forall \ x \in M \\
\lim_{x \rightarrow x^{\star}} \phi\left(x\right) = \phi\left(x^{\star}\right) = c
\end{gather}

\noindent then $\phi \left(x\right) = c$ for $x \in M$.

\end{lemma}

\begin{proof}

Since $\lambda = 0$, for $x \in M$, 

\begin{gather}
\phi\left(F^t\left(x\right)\right) = \phi\left(x\right)  \ \forall \ t
\end{gather}

\noindent so $\phi\left(x\right)$ is constant along any trajectory in $M$ (and thus an invariant of the trajectory).  The stability of $x^{\star}$ and continuity of $\phi\left(x\right)$ at $x^{\star}$, moreover, imply that

\begin{gather}
\phi\left(x\right) = \lim_{t\rightarrow \infty} \phi\left(F^t\left(x\right)\right) = \lim_{x \rightarrow x^{\star}} \phi \left(x\right) \\
\Rightarrow \phi\left(x\right) = \phi\left(x^{\star}\right) = c
\end{gather}

\end{proof}

For trajectory invariants (i.e., eigenfunctions with zero eigenvalues associated with them), continuity at the critical point implies continuity over the whole basin of attraction.

\begin{theorem}
\label{thm:constantCont}
Suppose $\mc{K}^t \phi = \phi$ (i.e., $\lambda = 0$), and suppose there exist isolated fixed points $x_A$ and $x_B$ such that $\phi(x_A) \neq \phi(x_B)$. Then $\phi\left(x\right)$ is discontinuous.
\end{theorem}

\begin{proof}
If $\phi\left(x\right)$ is continuous at its fixed points, then Lemma \ref{basin invariance}, $\phi\left(x\right)$ is constant on the fixed points' respective basins of attraction. As such, $\phi(\Re^n)$ is a countable set. In order for $\phi$ to be continuous, it has to be a globally constant function. Therefore, $\phi(x_A) \neq \phi(x_B)$ implies that $\phi$ is discontinuous.  

If, on the other hand, $\phi\left(x\right)$ is discontinuous at any of its fixed points, then it is tautologically true that $\phi\left(x\right)$ is discontinuous.
\end{proof}

\begin{theorem}
\label{closed traj thm}
Let $\mc{K}^t \phi = e^{\lambda t} \phi$ and $\relambda \neq 0$. Let 

\begin{gather}
M_x = \{p ~|~ F^t(x) = p,~t \in [-\infty, \infty]\}
\end{gather}

\noindent and let $cl(M_x)$ denote the closure of the set $M_x$. Suppose, $cl(M_x)$ is bounded and $\phi$ is continuous in $cl(M_x)$, then $\phi(x) = 0~\forall \ x \in cl(M_x)$. 
\end{theorem}

\begin{proof}
Since $cl(M_x)$ is a closed and bounded set and $\phi$ is continuous on $cl(M_x)$, $\phi$ is bounded on $cl(M_x)$. If $\phi$ is non-zero for any point $x \in cl(M_x)$, $[\mc{K}^t \phi](x) = e^{\lambda t} \phi(x)$ diverges as $t \rightarrow -\infty$ (for $\relambda < 0$) or as $t \rightarrow \infty$ (for $\relambda > 0$).  This contradicts the boundedness of $\phi$ in $cl(M_x)$. Therefore, $\phi(x) = 0~\forall \ x \in cl(M_x)$
\end{proof}

This implies that eigenfunctions with real eigenvalues evaluate to $0$ around closed trajectories.  If $cl(S_x)$ separates basins of attraction on which $\phi\left(x\right) \neq 0$, in general, then we would expect $\phi\left(x\right)$ to be discontinuous where those basins of attraction meet $cl(S_x)$.  Corollary \ref{zero phi cor} supports a similar conclusion for an eigenfunction with $\relambda < 0$ if the basins of attraction are separated by the unstable manifold of a different critical point.


\begin{theorem}
\label{thm:contEq}
Let $\mc{K}^t \phi = e^{\lambda t} \phi$ and $\relambda < 0$. Let $x^{\star}$ be an hyperbolic equilibrium point (i.e no center manifold). Let 

\begin{gather}
M(x^{\star}) = \{p ~|~ \lim_{t \rightarrow \infty} F^t(p) = x^{\star} \} \\
U(x^{\star}) = \{p ~|~ \lim_{t \rightarrow -\infty} F^t(p) = x^{\star}, \lim_{t \rightarrow \infty}F^t(p) \neq \infty \}
\end{gather}

\noindent represent the stable and the unstable manifold associated with $x^{\star}$. Suppose $\phi$ is $0$ in $U(x^{\star})$ (in accordance with Corollary \ref{zero phi cor}) and is uniformly continuous when restricted to basins of attraction. Then $\phi$ is continuous at $x^{\star}$.
\end{theorem}

\begin{proof}
Let $\epsilon > 0$. Let $x_n \rightarrow x^{\star}$. Since $\phi$ is uniformly continuous in basins of attraction, $\phi$ is continuous on $M(x^{\star})$. As such, we can assume that $x_n \notin M(x^{\star})$ without loss of generality. We can also assume $x_n \in A$ where $A$ is a basin of attraction for another critical point. In order to prove continuity of $\phi$ at $x^{\star}$, we need to show that $\left|\phi(x_n)\right|$ converges to $0$, since by Lemma \ref{fp zero lemma}, $\phi\left(x^{\star}\right) = 0$.

Let $\alpha_m \rightarrow 0$ and 

\begin{gather}
\lim_{n \rightarrow \infty} x_n + \alpha_m \in U(x^{\star}) \cap A \ \forall \ m
\end{gather}

Then,

\begin{align*}
\left|\phi(x_n)\right| &= \left|\phi(x_n) + \phi(x_n + \alpha_m) - \phi(x_n + \alpha_m)\right| \\
&\leq \left|\phi(x_n) - \phi(x_n + \alpha_m) \right| +  \left| \phi(x_n + \alpha_m)\right| 
\end{align*}

Since $\phi$ is piecewise continuous in basins of attraction,  there exist $M$ and $N$ such that 

\begin{gather}
\left|\phi(x_n) - \phi(x_n + \alpha_m) \right| \leq \dfrac{\epsilon}{2}\ \forall \ m \geq M \\
\left|\phi(x_n + \alpha_m)\right| \leq \dfrac{\epsilon}{2} \ \forall \ n \geq N
\end{gather}

\noindent since $\phi\left(x\right) = 0$ for $x \in U(x^{\star})$.  Therefore $\forall \ n \geq N$, we have $\left|\phi(x_n)\right|  \leq \epsilon$, which proves continuity at $x^\star$.

\end{proof}

In general, it is difficult to establish continuity of Koopman eigenfunctions across basins of attraction. This is evident in the restrictive assumptions (for e.g uniform continuity of eigenfunctions when restricted to basins) made in Theorem \ref{thm:contEq} in to establish continuity at a hyperbolic equilibrium point.

\subsection{CONTROL LIMITATIONS}

These continuity results have implications for controllability.  To draw out those implications, we refer back to some previous results in the literature.  Firstly, we consider the control formulation

\begin{gather}
\dot{\psi}_x = L_x \psi_x \left(x\right) + L_{xu} \psi_{xu} \left(x,u\right) \\
\psi_{xu}\left(x,0\right) = 0
\end{gather}

\noindent where $\psi_x\left(x\right)$ and $\psi_{xu} \left(x,u\right)$ are vectors of (invariant) Koopman observables.  There are other formulations present in the literature, but not all of those formulations are consistent \cite{bakker19kr} (i.e., they do not obey the chain rule).  Secondly, if an uncontrolled dynamical system $\dot{x} = f\left(x,0\right)$ has multiple isolated critical points, then its Koopman representation is rank deficient ($L_x$ is rank deficient) or the observables are identically zero at each critical point ($\psi_x\left(x^{\star}\right) = 0 \ \forall \ x^{\star}$) \cite{bakkerCDC2019}.  If the observables span the state, then $L_x$ is rank deficient, meaning that there is at least one eigenfunction $\phi\left(x\right)$ such that $\lambda = 0$.

\begin{theorem}
\label{control thm}
If $\phi_{x,0}\left(x_A\right) \neq \phi_{x,0} \left(x_B\right)$, $\lambda = 0$, for two fixed points $A$ and $B$, if $\phi_{x,0}\left(x\right)$ is continuous at all fixed points, and if $\psi_{xu} \left(x,u\right)$ is bounded, then it is not possible to reach the basin of attraction of $x_A$ from the basin of attraction of $x_B$.
\end{theorem}

\begin{proof}
Let $M_A$ and $M_B$ be the sets corresponding to the basins of attraction of $x_A$ and $x_B$, respectively.  By Lemma \ref{basin invariance}, $\phi\left(x\right)$ is constant on all basins of attraction and thus piecewise constant over the whole domain of interest.  Since $\phi_{x,0}\left(x_A\right) \neq \phi_{x,0}\left(x_B\right)$ and $\phi_{x,0}\left(x\right)$ is piecewise constant, $\phi_{x,0}\left(x\left(t\right)\right)$ will be discontinuous for any path $x\left(t\right)$ that passes through both $M_A$ and $M_B$.

Consider the eigendecomposition of the Koopman control formulation

\begin{gather}
L_x = Q D Q^{-1} \\
\phi_x\left(x\right) = Q^{-1} \psi_x\left(x\right) \\
\dot{\phi}_x = D \phi_x\left(x\right) + Q^{-1} L_{xu} \psi_{xu} \left(x,u\right)
\end{gather}

\noindent where $\phi_x\left(x\right)$ is the vector of eigenfunctions of the original uncontrolled system, and $D$ is the diagonal matrix with the corresponding eigenvalues as its entries.  Note that for the null eigenfunction $\phi\left(x\right)$, $\lambda = 0$, and thus $\dot{\phi}_{x,0}$ only depends on $Q^{-1}$, $L_{xu}$, and $\psi_{xu} \left(x,u\right)$.   To generate an instantaneous jump in the null eigenfunction, it is necessary to produce an infinite $\dot{\phi}_{x,0}$.  This is not possible, however, because $\psi_{xu}\left(x,u\right)$ is bounded.  Therefore, the control observable $\psi_{xu}\left(x,u\right)$ cannot generate a trajectory from $M_B$ to $M_A$.

\end{proof}

\begin{corollary}
If $\phi_{x,0}\left(x_A\right) \neq \phi_{x,0} \left(x_B\right)$, $\lambda = 0$, for two fixed points $A$ and $B$, if $\phi_{x,0}\left(x\right)$ is continuous at all fixed points, and if $\psi_x \left(x\right)$ is bounded, then it is not possible to reach the basin of attraction of $x_A$ from the basin of attraction of $x_B$ by using a feedback control of the form

\begin{gather}
\psi_{xu} \left(x,u\right) = -F \psi_x \left(x\right)
\end{gather}

\end{corollary}

Theorem \ref{control thm} and its corollary consider a specific Koopman control formulation, but the same proofs apply, \textit{mutatis mutandis}, to the other formulations considered in Bakker et al. \cite{bakker19kr}.  The same proofs also apply if eigenfunctions with non-zero eigenvalues are discontinuous across boundaries of basins of attraction, though it is more difficult to specify when this will and will not be the case; see the comments following Theorem \ref{closed traj thm} on this.  This does not mean that there are no control trajectories for the original system $\dot{x} = f\left(x,u\right)$ that can traverse the boundaries of basins of attraction.  Rather, it simply means that the Koopman observables corresponding to those trajectories may be unbounded.

\section{DISCUSSION}
\addtolength{\textheight}{-2.5cm} 

\label{sec:discussion}
The proofs and corollaries listed above have implications both for learning KO representations in general and for using the KO for calculating optimal control policies.  In producing the control results, we assume a consistent formulation, but if the formulation is not consistent, then there may be more serious problems with the Koopman representation (see \cite{bakker19kr}).  Firstly, the results show that, from an analysis perspective, it is advantageous to restrict the Koopman observables to functions that are `well-behaved'.  Functions with asymptotes, for example, are potentially problematic: they may go to infinity within the domain of interest, which violates the boundedness conditions used in many of the proofs, but asymptotic decay to 0 as $x \rightarrow \infty$ can also make it harder to guarantee stability (e.g., see Corollary \ref{cor:cor11}).

Secondly, using continuous observables (as in EDMD) is sufficient for use within basins of attraction (and justified by Hartman-Grobman type theorems \cite{mezic2017koopman}), but when there are multiple isolated critical points (with multiple basins of attraction), Koopman invariant functions are likely to have discontinuities at the boundaries.  Theorem \ref{closed traj thm} suggests that this could happen if the basin in question is bounded by periodic orbits.  Similarly, if the eigenfunctions are supported on the basin of attraction, Corollary \ref{zero phi cor} gives conditions under which those eigenfunctions would not be supported on the boundary, and that could also produce discontinuities at the boundary of the basin of attraction.

For multi-modal systems, it will likely be difficult to represent Koopman eigenfunctions (and/or Koopman observables) with a pre-set dictionary of basis functions.  For small, well-understood systems, the relevant basins of attraction may already be known and well-defined, but for a black box system, possibly with many dimensions, such information is not likely to be available \textit{a priori}.  In such situations, it would be easy to sample trajectories from different basins of attraction without realizing it.  The learning process could fail badly in such situations.  Neural networks are better able to approximate discontinuous functions, and thus they may be better suited for black box KO learning.

Thirdly, using control with the Koopman operator will, in general, only produce local optima -- namely, the local optimum that corresponds to the basin of attraction in which the trajectory begins.  In a way, this makes sense: it is not possible to get a globally optimal solution by using what is essentially a local optimization technique.  Attempting to learn control strategies that make it possible to move between basins of attraction will be challenging in and of itself.  Representing singularities (like the Dirac-delta function) is much harder than representing piecewise discontinuities -- even with a neural network.  Using discrete-time sampling, as is commonly done, also may not accurately represent the error in that region of the state space (see the example in \cite{bakkerCDC2019}).

Overall, these results show the limitations not of the KO itself but of our ability to calculate it and its observables from data in general nonlinear systems.  In other words, numerical KO representations may suffer irreducible model bias not just because of an inconsistent representation \cite{bakker19kr}, which could be remedied, or because of a finite number of observables, which could in principle be expanded, but because we lack the ability to represent the necessary observables.  This leads to two recommendations for future work.  Firstly, it may be worthwhile to study the use of observables that are not `well-behaved' (e.g., Dirac-delta functions, unit step functions, etc.).  Secondly, future work should build upon the stability results described above to account for the effects of error in the representations.  In particular, it would be worth studying how the error associated with the calculated KO representation affects our ability to guarantee stability (of the original system) from the KO.

\section{Conclusions}

In this paper, we provided a brief survey of some key analytical results wherein information contained in the Koopman operator of a dynamical system is used to infer properties of the original (nonlinear) dynamical system.  Many of these results relied on eigenfunction continuity assumptions.  Here, we provided a set of analytical results regarding stability that did not rely on these assumptions.  Following this, we defined some conditions under which eigenfunctions would be (dis)continuous or identically zero; those sets on which the eigenfunctions are not supported may constitute the boundaries of basins of attraction and may thus create discontinuities at those boundaries.  Those discontinuities then affect controllability and, in general, will limit Koopman control trajectories to the basin of attraction (of the uncontrolled dynamical system) in which they begin.  All of this, in turn, has implications for the nature and learnability of Koopman observables corresponding to nonlinear control systems with multiple basins of attraction.

\section{ACKNOWLEDGMENTS}

The research described in this paper was funded by the Deep Learning for Scientific Discovery Laboratory Directed Research and Development Investment at the Pacific Northwest National Laboratory, a multiprogram national laboratory operated by Battelle for the U.S. Department of Energy.


\bibliography{bib}

\begin{thebibliography}{10}

\bibitem{budisic12jsr}
M.~Budi{\v{s}}i{\'c}, R.~Mohr, and I.~Mezi{\'c}, ``Applied koopmanism,'' {\em
  Chaos: An Interdisciplinary Journal of Nonlinear Science}, vol.~22, no.~4,
  p.~047510, 2012.

\bibitem{tu14jsr}
J.~H. Tu, C.~W. Rowley, D.~M. Luchtenburg, S.~L. Brunton, and J.~N. Kutz, ``On
  dynamic mode decomposition: Theory and applications,'' {\em Journal of
  Computational Dynamics}, vol.~1, no.~2, pp.~391--421, 2014.

\bibitem{williams15jsr}
M.~O. Williams, I.~G. Kevrekidis, and C.~W. Rowley, ``A data--driven
  approximation of the koopman operator: Extending dynamic mode
  decomposition,'' {\em Journal of Nonlinear Science}, vol.~25, no.~6,
  pp.~1307--1346, 2015.

\bibitem{kutz16jsr}
J.~N. Kutz, J.~L. Proctor, and S.~L. Brunton, ``Koopman theory for partial
  differential equations,'' {\em arXiv preprint arXiv:1607.07076}, 2016.

\bibitem{proctor18jsr}
J.~L. Proctor, S.~L. Brunton, and J.~N. Kutz, ``Generalizing koopman theory to
  allow for inputs and control,'' {\em SIAM Journal on Applied Dynamical
  Systems}, vol.~17, no.~1, pp.~909--930, 2018.

\bibitem{williams16cp}
M.~O. Williams, M.~S. Hemati, S.~T. Dawson, I.~G. Kevrekidis, and C.~W. Rowley,
  ``Extending data-driven koopman analysis to actuated systems,'' in {\em IFAC
  Symposium on Nonlinear Control Systems (NOLCOS)}, 2016.

\bibitem{korda18jsr}
M.~Korda and I.~Mezi{\'c}, ``Linear predictors for nonlinear dynamical systems:
  Koopman operator meets model predictive control,'' {\em Automatica}, vol.~93,
  pp.~149--160, 2018.

\bibitem{yeung17jsr}
E.~Yeung, S.~Kundu, and N.~Hodas, ``Learning deep neural network
  representations for koopman operators of nonlinear dynamical systems,'' {\em
  arXiv preprint arXiv:1708.06850}, 2017.

\bibitem{li17jsr}
Q.~Li, F.~Dietrich, E.~M. Bollt, and I.~G. Kevrekidis, ``Extended dynamic mode
  decomposition with dictionary learning: A data-driven adaptive spectral
  decomposition of the koopman operator,'' {\em Chaos: An Interdisciplinary
  Journal of Nonlinear Science}, vol.~27, no.~10, p.~103111, 2017.

\bibitem{arbabi17jsr}
H.~Arbabi and I.~Mezic, ``Ergodic theory, dynamic mode decomposition, and
  computation of spectral properties of the koopman operator,'' {\em SIAM
  Journal on Applied Dynamical Systems}, vol.~16, no.~4, pp.~2096--2126, 2017.

\bibitem{mezic15cp}
I.~Mezi{\'c}, ``On applications of the spectral theory of the koopman operator
  in dynamical systems and control theory,'' in {\em Decision and Control
  (CDC), 2015 IEEE 54th Annual Conference on}, pp.~7034--7041, IEEE, 2015.

\bibitem{mauroy13jsr}
A.~Mauroy, I.~Mezi{\'c}, and J.~Moehlis, ``Isostables, isochrons, and koopman
  spectrum for the action--angle representation of stable fixed point
  dynamics,'' {\em Physica D: Nonlinear Phenomena}, vol.~261, pp.~19--30, 2013.

\bibitem{mauroy16jsr}
A.~Mauroy and I.~Mezi{\'c}, ``Global stability analysis using the
  eigenfunctions of the koopman operator,'' {\em IEEE Transactions on Automatic
  Control}, vol.~61, no.~11, pp.~3356--3369, 2016.

\bibitem{mezic05jsr}
I.~Mezi{\'c}, ``Spectral properties of dynamical systems, model reduction and
  decompositions,'' {\em Nonlinear Dynamics}, vol.~41, no.~1-3, pp.~309--325,
  2005.

\bibitem{bakker19kr}
C.~Bakker, S.~Rosenthal, and K.~Nowak, ``Koopman representations of dynamic
  systems with control,'' {\em arXiv preprint arXiv:1908.02233}, 2019.

\bibitem{bakkerCDC2019}
C.~Bakker, K.~Nowak, and S.~Rosenthal, ``Learning koopman operators for systems
  with isolated critical points,'' in {\em 58th IEEE Conference on Decision and
  Control (CDC)}, p.~Accepted, IEEE, 2019.

\bibitem{mezic2017koopman}
I.~Mezic, ``Koopman operator spectrum and data analysis,'' {\em arXiv preprint
  arXiv:1702.07597}, 2017.

\end{thebibliography}
\bibliographystyle{ieeetr}

\end{document}